\documentclass{article}
\usepackage{latexsym, amsthm, amsmath, bbm}
\usepackage{amssymb}
\usepackage{amsfonts}
\usepackage{amstext}
\usepackage{graphicx}
\usepackage{multicol}
\usepackage{subfigure}
\usepackage{textcomp}
\usepackage{enumerate}
\newtheorem{Theorem}{Theorem}
\newtheorem{Lemma}{Lemma}
\newtheorem{Proposition}{Proposition}
\newtheorem{Definition}{Definition}
\newtheorem{Corollary}{Corollary}

\theoremstyle{remark}
\newtheorem{Remark}{Remark}

\theoremstyle{remark}

\theoremstyle{plain}

\theoremstyle{remark}
\newtheorem{Example}{Example}
\newcommand{\R}{{\mathbb{R}}}

\newcommand{\Z}{{\mathbb{Z}}}

\newcommand {\N}{{\mathbb{N}}}

\newcommand{\s}{{\mathbb S}}

\begin{document}

\title{A General Method to Determine Limiting Optimal Shapes for Edge-Isoperimetric Inequalities}

\author{Emmanuel Tsukerman \thanks{Supported by the National Science Foundation Graduate Research Fellowship under Grant No.  DGE 1106400.  Any opinion, findings, and conclusions or recommendations expressed
in this material are those of the authors(s) and do not necessarily reflect the views of the National Science
Foundation.} \\
\small University of California, Berkeley \\
\and
Ellen Veomett \\
\small Saint Mary's College of California}

\maketitle

\begin{abstract}
For a general family of graphs on $\Z^n$, we translate the edge-isoperimetric problem into a continuous isoperimetric problem in $\R^n$.  We then solve the continuous isoperimetric problem using the Brunn-Minkowski inequality and Minkowski's theorem on Mixed Volumes.  This translation allows us to conclude, under a reasonable assumption about the discrete problem, that the shapes of the optimal sets in the discrete problem approach the shape of the optimal set in the continuous problem as the size of the set grows.  The solution is the zonotope defined as the Minkowski sum of the edges of the original graph.  

We demonstrate the efficacy of this method by revisiting some previously solved classical edge-isoperimetric problems.   We then apply our method to some discrete isoperimetric problems which had not previously been solved.   The complexity of those solutions suggest that it would be quite difficult to find them using discrete methods only.
\end{abstract}

\section{Introduction}

For a space with some notion of ``volume'' and ``boundary'', an isoperimetric inequality gives an upper bound on the volume of a set of fixed boundary.  For example, one can consider Euclidean space $\R^n$ where ``volume'' is the usual notion of Lebesgue measure, and ``boundary'' is the usual notion of the boundary.  That is, for $X\subset \R^n$, the boundary of $X$ is defined:
\begin{equation*}
 \lim_{\epsilon \to 0^+} \frac{\text{Vol}\left(X+\epsilon B\right) - \text{Vol}(X)}{\epsilon}
\end{equation*}
where $B$ is the Euclidean ball of radius 1 and $X+\epsilon B$ refers to the Minkowski sum:
\begin{align*}
B &= \{x \in \R^n: ||x||_2 \leq 1\} \\
X + \epsilon B &= \{x+ \epsilon y: x \in X, y \in B\}
\end{align*}

The well-known Euclidean isoperimetric inequality states that among all sets with a fixed boundary, the corresponding Euclidean ball has the greatest volume.  This is equivalent to saying that among all sets with a fixed volume, the corresponding Euclidean ball has the smallest boundary.  

One can similarly define an isoperimetric inequality for any graph.  Given a  simple undirected graph $G = (V,E)$, we say that the \emph{volume} of a set $S \subset V$ is simply the number of vertices in that set: $|S|$.  The \emph{boundary} of that set can be calculated in one of two ways: using the edge boundary or the vertex boundary.

\begin{Definition}
The \emph{vertex boundary} $\partial(S)$  of a set $S \subset V$ is the set of vertices in $V \backslash S$ which are adjacent to some vertex in $S$:
\begin{equation*}
\partial(S) = \{v \in V\backslash S: (v,u) \in E \text{ for some } u \in S\}
\end{equation*}
Thus, the \emph{size} of the vertex boundary is $|\partial(S)|$.  

The \emph{edge boundary} $\partial_e(S)$ of a set $S \subset V$ is the set of edges $(u,v) \in E$ ``exiting'' the set $S$:
\begin{equation*}
\partial_e(S) = \{(u,v) \in E: \left| \{u,v\} \cap S\right|= 1\}
\end{equation*}
Thus, the \emph{size} of the edge boundary is $|\partial_e(S)|$.  
\end{Definition}

In the discrete case, the isoperimetric inequality is usually stated in terms of fixing the volume and finding the set of smallest boundary.

Both vertex and edge-isoperimetric inequalities on graphs have been studied for various families of graphs.  Vertex-isoperimetric inequalities are studied, for example, in \cite{DiscTor, MR1082843, MR0200192, MR0168489, MR1612869, MR2946103} and edge-isoperimetric inequalities in 
\cite{MR2346808, SamorodnitskyHammingCube,  MR1137765, MR1909858, MR3436548}.   Some general techniques for solving discrete isoperimetric inequalities have been developed, including compression and stabilization \cite{MR2035509}.  

While most of the papers on discrete isoperimetric inequalities study the discrete problems directly,  in \cite{MR1137765} the authors use a continuous formulation of the discrete question to solve the discrete problem.  In this paper, we discuss a general method which can be used to translate a discrete isoperimetric inequality into a continuous one.  We then solve the continuous isoperimetric inequality, and apply this technique to both graphs whose isoperimetric inequality was previously known and graphs whose isoperimetric inequality was not previously known.

More specifically, we introduce the following definition:
\begin{Definition}
A simple graph  $G = (V,E)$ is called a PL graph (Primitive Lattice graph) if it satisfies the following: 
\begin{itemize}
\item   $V = \Z^n$
\item There exist integer vectors $v_1, v_2, \dots, v_k$ (with $v_i \not= -v_j$ for any $i, j$) such that for any $u \in \Z^n$ the edges in $E$ involving $u$ are precisely the edges:
\begin{equation*}
(u,u\pm v_1), (u, u\pm v_2), (u, u\pm v_3), \dots, (u, u\pm v_k)
\end{equation*}
\item For each integer vector  $v_i = (v_{i1}, v_{i2}, \dots, v_{in})$ above, the entries $\{v_{i1}, v_{i2}, \dots, v_{in}\}$ are relatively prime (primitive). 
\item The span of $\{v_1, v_2, \dots, v_k\}$ is $\R^n$.
\end{itemize}
\end{Definition}
We note that the above conditions imply that $G$ is regular of degree $2k$, and any translation mapping $\Z^n$ to itself is an isomorphism of this graph to itself.  The last condition implies that the graph is ``full dimensional'' and appropriately lives in $\Z^n$ (as opposed to $\Z^\ell$ for some $\ell<n$).  

For any PL graph, we also define the following:
\begin{Definition}
Suppose $G = (V,E)$ is a PL graph whose edges are given by the vectors $v_1, v_2, \dots, v_k$.  Then the \emph{edge segments} $\ell_i, i=1, 2, \dots, k$ of $G$ are the line segments from the origin to $v_i$ for each $i$:
\begin{equation*}
\ell_i = \{t \vec{0} + (1-t) v_i: t \in [0,1]\}
\end{equation*}
\end{Definition}

We now have the following Lemma, which will be proved in Subsection \ref{BoundarySubSection}:

\begin{Lemma}\label{ImportantLemma}
Let $G = (\Z^n,E)$ be a PL graph.  Let $\ell_1, \ell_2, \dots, \ell_k$ be the edge segments of $G$.  Let Z be the zonotope
\begin{equation*}
Z = \sum_{i=1}^k \left(-\ell_i + \ell_i\right)
\end{equation*}
where all sums are the Minkowski sum.  Let $\mu_n$ denote the Lebesgue measure on $\R^n$ and let $b$ be the real-valued function on sets $A \subset \R^n$ defined by:
\begin{equation*}
b(A) = \lim_{\epsilon \to 0^+} \frac{\mu_n\left(A+\epsilon Z \right) - \mu_n\left(A\right)}{\epsilon}
\end{equation*}
where $A+\epsilon Z$ is the Minkowski sum.

Then for any convex set $X \subset \R^n$, 
 we  have for $\alpha \in \R$
\begin{align*}
\lim_{\alpha \to \infty}\frac{ \mu_n(\alpha X)}{|\Z^n \cap \alpha X|}  &= 1 \quad \quad \quad \text{ and } \\
\lim_{\alpha \to \infty}\frac{b(\alpha X )}{|\partial_e(\Z^n \cap \alpha X)|}  &= 1
\end{align*}
\end{Lemma}

The above Lemma tells us that solving the isoperimetric inequality on $\R^n$ using the boundary function $b$ should give us an idea of the shape of set that solves the edge isoperimetric inequality for a PL graph.  (We discuss conditions under which $X$ is guaranteed to be the optimal shape in Remark \ref{LimitingRemark}).  The main Theorem of this paper is that we can solve the corresponding continuous isoperimetric inequality on $\R^n$ using boundary function $b$:

\begin{Theorem}\label{DiscToContTheorem}
Suppose $G = (\Z^n, E)$ is a PL graph with edge segments $\ell_1, \ell_2, \dots, \ell_k$.  Let $Z$ be the zonotope
\begin{equation*}
Z = \sum_{i=1}^k \left(-\ell_i + \ell_i\right)
\end{equation*}
where all sums are the Minkowski sum.   Let $X$ be a scaling of $Z$ and $b$ be the boundary function as defined in Lemma \ref{ImportantLemma}.  Then for any $A \subset \R^n$ with $\text{Vol}(A) = \text{Vol}(X)$, we have
\begin{equation*}
b(X) \leq b(A)
\end{equation*}
with equality if and only if $A$ is homothetic to $X$.
\end{Theorem}

The paper is organized as follows: in Section \ref{ContinuousSection} we prove Lemma \ref{ImportantLemma} and Theorem \ref{DiscToContTheorem}.  In section \ref{KnownSection} we apply Theorem \ref{DiscToContTheorem} to cases where the discrete isoperimetric inequality has been solved, to show how easily the proper ``shape'' of the discrete solution can be found.  And in section \ref{UnknownSection} we apply Theorem \ref{DiscToContTheorem} to some cases where the discrete isoperimetric problem has \emph{not} previously been solved.

\section{Defining and Solving the Continuous Isoperimetric Problem}\label{ContinuousSection}


\subsection{Limiting Solutions}\label{AssumptionsSubSection}

We expect our technique will find the shape of sets with minimum edge boundary in PL graphs which have a \emph{limiting solution}:

\begin{Definition}
Suppose that $G = (V,E)$ is a PL graph.  We say that the edge isoperimetric problem for $G$ has a \emph{limiting solution} with convex body $K$ if for each $m \in \N$, there exists a set of minimum edge boundary $S_m\subset \Z^n$ of size $m$ such that the following holds: 

There exists a subsequence $S_{k_n}$ and a function $f:\Z \to \R$ with $\lim_{n \to \infty} f(k_n) = \infty$ such that 
\begin{equation*}
S_{k_n} = \Z^n \cap  f(k_n)K
\end{equation*}
where $f(k_n)K$ is the scaling of the set $K$ by the number $f(k_n)$.
\end{Definition}

%
 In  words, a PL graph has a limiting solution if for any $N \in \N$, we can find a set of minimum boundary with volume larger than $N$ such that the set consists of the integer points in a scaling of a fixed convex body.

It is reasonable to expect a PL graph to have a limiting solution because a PL graph is so symmetric.  It is natural to expect shapes for sets of minimum boundary to be nearly the points in a convex set, and precisely the points in a convex set for particular volumes.  And again, given the symmetry, it is reasonable to expect that the same optimal shape appears repeatedly as the volume grows.  We also note that in Section \ref{KnownSection}, we apply our technique to all PL graphs we could find in the literature for which the edge-isoperimetric inequality has already been solved.  All of them do indeed satisfy this assumption.

Note that this assumption implies that
\begin{equation*}
\lim_{n \to \infty}\frac{\text{Vol}\left(f(k_n)K\right)}{|S_{k_n}|} = 1
\end{equation*}
Indeed, from Lattice theory, we know that as a convex set is scaled by an unbounded factor, the number of integer points in the scaled set approaches the volume of the scaled set.  (One can prove this by, for example, modifying the proof of Theorem 2.3 in Chapter VII Section 2 of \cite{MR1940576} and also using Minkowski's theorem on mixed volumes, which is stated as Theorem \ref{MixedVolumeThm} below).  

We also note that this assumption implies that all of the sets $S_{k_n}$  have no ``gaps'':

\begin{Definition}
Let $(\Z^n,E)$ be an PL graph with edges corresponding to vectors $v_1, v_2, \dots, v_k$.    For $S \subset \Z^n$, we define 
\begin{equation*}
\text{gap}_{v_i}(S) = \{x \in \Z^n: x-v_i \in S, x \not\in S, \text{ and } x+bv_i \in S \text{ for some }b \geq 1\}
\end{equation*}
Thus, one can think of a point $x \in \text{gap}_{v_i}(S)$ as the first vertex in $\Z^n$ which indicates a gap in $S$ in the line through $x$ in the direction of $v_i$.

We say that $S \subset \Z^n$ has \emph{no gaps} if for each $i=1, 2, \dots, k$ the set $\text{gap}_{v_i}(S)$ is empty.
\end{Definition}

%

\subsection{Appropriate Boundary Definition for Continuous Problem}\label{BoundarySubSection}

It is not too hard to calculate the edge boundary for a general set $S \subset \Z^n$ in a PL graph.  First we require a definition.

\begin{Definition}

Let $(\Z^n, E)$ be a PL graph with edges corresponding to vectors $v_1, v_2, \dots, v_k$.  We define $P_{v_i}(S)$ to be the projection of $S$ onto the hyperplane of $\R^n$ which is perpendicular to $v_i$.  That is,
\begin{equation*}
P_{v_i}(S) = \left\{u-\frac{\left<u,v_i\right>}{||v_i||_2}v_i: u \in S\right\}
\end{equation*}

\end{Definition}

We can now calculate the edge boundary of $S \subset \Z^n$:

\begin{Theorem}\label{GapThm}
Let $(\Z^n,E)$ be a PL graph with edges corresponding to vectors $v_1, v_2, \dots, v_k$.  Let $S \subset \Z^n$ be a finite set.  Then
\begin{equation}\label{EqnWithGap}
|\partial_e (S)| = 2\sum_{i=1}^k\left( |P_{v_i}(S)| +|\text{gap}_{v_i}(S)|\right)
\end{equation}
\end{Theorem}

\begin{proof}
We proceed by induction on $|S|$.  If $|S| =1$, then $\text{gap}_{v_i}(S) = \emptyset$ for each $i=1, 2, \dots, k$.  We can also see that if $S = \{u\}$, then
\begin{equation*}
\partial_{e}( S) = \{(u, u+v_i): i=1, 2, \dots, k\} \cup \{(u,u-v_i), i=1, 2, \dots, k\}
\end{equation*}
Additionally, in this case, 
\begin{equation*}
|P_{v_i}(S)| = 1
\end{equation*}
for each $i=1, 2, \dots, k$.  Thus, we have 
\begin{equation*}
|\partial_e (S)| = 2\sum_{i=1}^k\left( |P_{v_i}(S)| +|\text{gap}_{v_i}(S)|\right)
\end{equation*}
if $|S|=1$.

Now suppose that $|S|>1$.  Fix $u \in S$.  By induction,
\begin{equation*}
|\partial_e (S\backslash\{u\})| = 2\sum_{i=1}^k\left( |P_{v_i}(S\backslash\{u\})| +  |\text{gap}_{v_i}(S\backslash\{u\})|\right)
\end{equation*}
Consider what $u$ contributes to the edge boundary of $S$.  Note that for  each $v_i, i=1, 2, \dots, k$ we have three cases:

\noindent \underline{Case 1: Both $u+v_i$ and $u-v_i$ are in $S$}  In this case,
\begin{align*}
(u, u+v_i)  \in \partial_e(S\backslash \{u\}) & \quad\quad \quad   (u-v_i, u)  \in \partial_e(S\backslash \{u\})  \\
(u, u+v_i)  \not\in \partial_e(S) &  \quad\quad \quad   (u-v_i, u)  \not\in \partial_e(S)
\end{align*}
and
\begin{equation*}
\text{gap}_{v_i}(S) =  \text{gap}_{v_i}(S\backslash\{u\})\backslash\{u\} \quad \quad \quad P_{v_i}\left(S\backslash\{u\}\right) = P_{v_i}(S)
\end{equation*}
Thus we can see that when considering only edges in the $v_i$ direction, both the left and right hand sides of equation \eqref{EqnWithGap} go down by 2  when $u$ is added back to $S$.

\noindent \underline{Case 2: Exactly one of  $u+v_i$ or $u-v_i$ is in $S$}  Without loss of generality, assume that $u-v_i\in S$.  Then
\begin{align*}
(u-v_i\, u)  \in \partial_e(S\backslash \{u\}) & \quad\quad \quad   (u,u+v_i)  \not\in \partial_e(S\backslash \{u\})  \\
(u-v_i\, u)  \not\in \partial_e(S) &  \quad\quad \quad   (u,u+v_i)  \in \partial_e(S)
\end{align*}
and
\begin{equation*}
\left|\text{gap}_{v_i}(S)\right| =  \left|\text{gap}_{v_i}(S\backslash\{u\})\right| \quad \quad \quad  P_{v_i}\left(S\backslash\{u\}\right) = P_{v_i}(S)
\end{equation*}

Thus we can see that when considering only edges in the $v_i$ direction, both the left and right hand sides of equation \eqref{EqnWithGap} do not change  when $u$ is added back to $S$.

\noindent \underline{Case 3: Neither $u+\epsilon$ nor $u-\epsilon$ are in $S$}

In this case, 
\begin{align*}
(u, u+v_i)  \not\in \partial_e(S\backslash \{u\}) & \quad\quad \quad   (u-v_i, u)  \not\in \partial_e(S\backslash \{u\})  \\
(u, u+v_i)  \in \partial_e(S) &  \quad\quad \quad   (u-v_i, u)  \in \partial_e(S)
\end{align*}
and  \emph{either} $u+b v_i \in S$ for some $b$ in which case
\begin{equation*}
\left|\text{gap}_{v_i}(S)\right| =  \left|\text{gap}_{v_i}(S\backslash\{u\})\right|+1 \quad \quad \quad P_{v_i}\left(S\backslash\{u\}\right) = P_{v_i}(S)
\end{equation*}
\emph{or}  $u+b v_i \not\in S$ for any $b$ in which case
\begin{equation*}
\text{gap}_{v_i}(S)=  \text{gap}_{v_i}(S\backslash\{u\})= \emptyset \quad \quad \quad \left|P_{v_i}(S)\right| = \left|P_{v_i}\left(S \backslash\{u\}\right)\right| +1
\end{equation*}

Thus we can see that when considering only edges in the $v_i$ direction, both the left and right hand sides of equation \eqref{EqnWithGap} go up by 2  when $u$ is added back to $S$.

Since $i$ was arbitrary, we can see that all of the changes between $\partial_e(S \backslash\{u\})$ and $\partial_e(S)$ are balanced out by changes in either the corresponding gaps or projections.  Thus, we have
\begin{equation*}
|\partial_e (S)| = 2\sum_{i=1}^k\left( |P_{v_i}(S)| +|\text{gap}_{v_i}(S)|\right)
\end{equation*}

\end{proof}

Theorem \ref{GapThm} clearly has the following corollary:

\begin{Corollary}\label{EdgeBoundaryCorollary}
Let $(\Z^n,E)$ be a PL graph with edges corresponding to vectors $v_1, v_2, \dots, v_k$.   Let $S \subset \Z^n$ be a finite set such that $\text{gap}_{v_i}(S) = \emptyset$ for each $i=1, 2, \dots, k$.  Then
\begin{equation}\label{eqn1}
|\partial_e (S)| = 2\sum_{i=1}^k|P_{v_i}(S)| 
\end{equation}

\end{Corollary}

\begin{Remark}
We note that Corollary \ref{EdgeBoundaryCorollary} is a nice counterpoint to the vertex boundary calculations in \cite{MR2946103}, which calculate the vertex boundary of optimal sets in a particular PL graph as a weighted sum of projections of the graph.  
\end{Remark}

Recall that from our arguments in Subsection \ref{AssumptionsSubSection}, for a PL graph with a limiting solution, once the volume is large enough, the sets of minimum boundary have no gaps.  Thus, we can assume that for volume large enough, the boundary of a set of minimum boundary can be calculated using equation \eqref{eqn1}.  In order to finish our translation of the discrete isoperimetric problem into a continuous isoperimetric problem, we must now define an appropriate boundary function $b$ on $\R^n$.  Let $\mu_n$ denote the usual Lebesgue measure on $\R^n$.

\begin{Definition}
Let $u, P \subset \R^n$. Define
\begin{equation}\label{deriv}
D_u P=\lim_{\epsilon \rightarrow 0^{+}} \frac{\mu_n(P+\epsilon u)-\mu_n(P)}{\epsilon}.
\end{equation}
\end{Definition}

Note that the special case of $u = B_n$, the Euclidean ball in $\R^n$ of radius 1, gives the surface volume of the set $P$.  

We now need a couple of Lemmas, which we note also appear in \cite{TsukermanVeomettCauchy}.  We include them here for completeness.  We will use Minkowski's theorem on mixed volumes, which can be found, for example, in  Chapter 5 of Schneider's text \cite{MR3155183}.  This theorem says that the volume of a Minkowski sum of convex bodies can be written as a polynomial in the coefficients of that Minkowski sum, where the coefficients of the polynomial depend only on the convex bodies.  Specifically:  

\begin{Theorem}\label{MixedVolumeThm}
Suppose $K_1, K_2, \dots, K_m$ are convex bodies in $\R^n$.  Then
\begin{equation*}
\mu_n\left(\lambda_1K_1+\lambda_2 K_2+ \cdots + \lambda_m K_m\right) = \sum \lambda_{i_1}\lambda_{i_2} \cdots \lambda_{i_n} V\left(K_{i_1}, K_{i_2}, \dots, K_{i_n}\right)
\end{equation*}
where the sum on the left hand side is the Minkowski sum, and the sum on the right hand side is over all multisets of size $n$ whose elements are in the set $\{1, 2, \dots, m\}$.  The  functions $V$ are nonnegative, symmetric, and depend only on the convex bodies $K_{i_1}, K_{i_2}, \dots, K_{i_n}$.  For a fixed $n$-dimensional convex body $K$, $V(\underbrace{K, K, \dots, K}_{n \text { times}}) = \mu_n(K)$.
\end{Theorem}

From Minkowski's Theorem on mixed volume it follows that if $P$ and $u$ are convex, then $D_u P$ is linear in $u$:

\begin{Lemma}\label{LinearityLemma}
Suppose $P, u, v \subset \R^n$ with $P, u$, and $v$ convex bodies, and suppose $\alpha, \beta \in \R$.  Then
\begin{equation*}
D_{\alpha u+\beta v}(P) = \alpha D_u(P)+\beta D_v(P)
\end{equation*}
\end{Lemma}

\begin{proof}
By definition we have
\begin{equation*}
D_{\alpha u+\beta v}(P) = \lim_{\epsilon \to 0^+} \frac{\mu_n\left(P+\epsilon \left(\alpha u+\beta v\right)\right) - \mu_n(P)}{\epsilon} =  \lim_{\epsilon \to 0^+} \frac{\mu_n\left(P+\epsilon \alpha u+ \epsilon \beta v\right) - \mu_n(P)}{\epsilon} 
\end{equation*}
From Theorem \ref{MixedVolumeThm}, we can see that
\begin{equation*}
D_{\alpha u+\beta v}(P) = \alpha V(\underbrace{P, P, \dots, P}_{n-1 \text{ times}}, u) + \beta V(\underbrace{P, P, \dots, P}_{n-1 \text{ times}}, v) 
\end{equation*}
where $V$ represents the function in the statement of Theorem \ref{MixedVolumeThm}.  Similarly, one can easily see that
\begin{align*}
D_u(P) &=  V(\underbrace{P, P, \dots, P}_{n-1 \text{ times}}, u) \\
 D_v(P) &=  V(\underbrace{P, P, \dots, P}_{n-1 \text{ times}}, v) 
 \end{align*}
and our Lemma is proved.
\end{proof}

 For the following Lemma, we prove that the derivative we've defined in equation \eqref{deriv}  calculates the volume of the projection of a convex body in the case where $u$ is a segment of    length 1.  For a convex body $K \subset \R^n$ and 
$u \in \s^{n-1}$, we denote by $P_u(K)$ the projection of $K$ onto the $(n-1)$-dimensional subspace of $\R^n$ which is perpendicular to $u$.  

\begin{Lemma}\label{DerivativeLemma}
Let $u$ be a segment of length $1$ and $K \subset \R^n$ convex. Then
\begin{equation*}\label{projeq}
\mu_{n-1} \left(P_u(K)\right) = D_u K.
\end{equation*}
\end{Lemma}

\begin{proof}
Let $L_u$ be the set of lines parallel to $u$.  Note that each $l \in L_u$ corresponds uniquely to a single point in $l^\perp$, so that $L_u$ is isomorphic to $\R^{n-1}$ (and thus we can define the measure $\mu_{n-1}$ on $L_u$).  Then
\[
\mu_n(K)=\int_{l \in L_u} \mu_1(l \cap K) d \mu_{n-1}.
\]
Now,
\[
\mu_{n-1}\left(P_u(K)\right) =\int_{\substack{ l \cap K \neq \emptyset \\ l \in L_u }} 1 d \mu_{n-1}.
\]
For $\epsilon >0$, we have
\[
\mu_n\left(K+\epsilon u\right)-\mu_n\left(K\right)=\int_{l \in L_u}\left( \mu_1\left(l \cap (K+\epsilon u)\right)-\mu_1\left(l \cap K\right) \right)d\mu_{n-1}.
\]
Convexity implies that
\[
\mu_1\left(l \cap (K+\epsilon u)\right)-\mu_1\left(l \cap K\right) = \epsilon
\]
whenever $l$ intersects $K$. Therefore the last integral is equal to
\[
\int_{\substack{ l \cap K \neq \emptyset \\ l \in L_u }} \epsilon d \mu_{n-1}=\epsilon\mu_{n-1} \left(P_u(K)\right) 
\]
and hence,
\begin{equation*}
D_u K = \lim_{\epsilon \to 0^+}  \frac{\mu_n(K+\epsilon u)-\mu_n(K)}{\epsilon} = \lim_{\epsilon \to 0^+} \frac{\epsilon\mu_{n-1} \left(K|u^\perp\right) }{\epsilon} = \mu_{n-1} \left(P_u(K)\right) 
\end{equation*}
\end{proof}

We are now able to see how to define the boundary function $b$ for the continuous isoperimetric problem.  Recall that for a PL graph $(\Z^n, E)$ with edges corresponding to vectors $v_1, v_2, \dots, v_k$, the boundary of a set $S \subset \Z^n$ having no gaps is twice the sum of the number of points in $P_{v_i}$ for $i=1, 2, \dots, k$:  

\begin{equation*}
|\partial_e (S)| = 2\sum_{i=1}^k|P_{v_i}(S)| 
\end{equation*}

Since $P_{v_i}(S) = P_{-v_i}(S)$, we can re-write this edge calculation as: 

\begin{equation*}
|\partial_e (S)| = \sum_{i=1}^k\left(|P_{v_i}(S)| +|P_{-v_i}(S)| \right)
\end{equation*}

If the edge isoperimetric inequality for the PL graph has a limiting solution,  for the sets $S_{k_n} \subset \Z^n$ of minimum boundary, the integer points in the set $f(k_n)K$ are precisely $S_{k_n}$.  Thus, again from Lattice Theory, $c_i|P_{v_i}(S)|$ is a good approximation for $\mu_{n-1}(P_{v_i}(f(k_n)K))$, where $c_i$ is the determinant of the lattice $\Lambda = P_{v_i}(\Z^n)$.  Thus, using what we found in Lemma \ref{DerivativeLemma}, we will define our boundary function $b$ on $\R^n$ as follows:  for $A \subset \R^n$.
\begin{equation}\label{EqnWithConstant}
b(A) = \sum_{i=1}^k c_i(D_{u_i}(A) + D_{-u_i}(A))
\end{equation}
where $c_i$ is the determinant of the lattice $P_{v_i}(\Z^n)$ and $u_i $ is the segment of length 1 in the direction of $v_i$.  Thus, for the optimal sets $S_{k_n}$, we will have both 
\begin{align*}
\lim_{n \to \infty}\frac{\text{Vol}\left(f(k_n)K\right)}{|S_{k_n}|} &= 1 \quad \quad \quad \text{ and } \\
\lim_{n \to \infty} \frac{b(f(k_n)K)}{|\partial_e(S_{k_n})| } &= 1
\end{align*}

Note that by the same argument, for \emph{any} convex set $X$, we have for $\alpha \in \R$ 
\begin{align}\label{GeneralBoundaryEquations}
\lim_{\alpha \to \infty}\frac{ \mu_n(\alpha X)}{|\Z^n \cap \alpha X|}  &= 1 \quad \quad \quad \text{ and }  \nonumber \\
\lim_{\alpha \to \infty}\frac{b(\alpha X )}{|\partial_e(\Z^n \cap \alpha X)|}  &= 1
\end{align}

We have nearly finished the proof of Lemma \ref{ImportantLemma};
we need only simplify the expression in equation \eqref{EqnWithConstant}.  Specifically,  we need to calculate the constants $c_i$.  In other words: what is the determinant of the lattice $P_{v_i}(\Z^n)$?

In order to answer this, we need the following Lemmas:

\begin{Lemma}\label{ProjectionLatticeLemma}
Suppose $a \in \Z^n$ and $\Lambda = P_a\left(\Z^n\right)$.  Then 
\begin{equation*}
\Lambda^* = \{x \in \Z^n: \left<x,a \right> = 0\}
\end{equation*}
where $\Lambda^*$ denotes the dual lattice to $\Lambda$.
\end{Lemma}

\begin{Lemma}\label{GeneralDeterminantLemma}
Suppose $a = (a_1, a_2, \dots, a_n)$ where $\{a_1, a_2, \dots, a_n\}$ is a set of relatively prime integers.  Let 
\begin{equation*}
\Lambda = P_a\left(\Z^n\right)
\end{equation*}
Then 
\begin{equation*}
\det(\Lambda) = \frac{1}{\sqrt{a_1^2+a_2^2+ \cdots +a_n^2}}
\end{equation*}
\end{Lemma}

\begin{proof}[Proof of Lemma \ref{ProjectionLatticeLemma}]
Suppose $a \in \Z^n$ and $\Lambda = P_a\left(\Z^n\right)$.  Suppose $y \in  \{x \in \Z^n: \left<x,a \right> = 0\}$.  Pick any $z \in \Lambda$.  Then we know that $z = x- \frac{\left<x,a\right>}{\left<a,a\right>} a$ for some $x \in \Z^n$.  Note that
\begin{equation*}
\left<z,y\right>  = \left< x- \frac{\left<x,a\right>}{\left<a,a\right>} a,y\right> = \left<x,y\right> \in \Z
\end{equation*}
This proves that $ \{x \in \Z^n: \left<x,a \right> = 0\} \subset \Lambda^*$.

Now suppose that $y \in \Lambda^*$.  Certainly $y$ lives in the same vector space as $\Lambda$, so that $y \in \R^n$ and  $\left<y,a\right> = 0$.  Since $y \in \Lambda^*$, we know that for any $z \in \Lambda$, $\left<y,z \right> \in \Z$.  For $i = 1, 2, \dots, n$, let $z_i$ be the vector which is:
\begin{equation*}
z_i = e_i - \frac{\left<e_i,a\right>}{\left<a,a\right>}a
\end{equation*}
Then note that if $y = (y_1, y_2, \dots, y_n)$ we must have
\begin{equation*}
\left<y,z_i\right> = \left< y, e_i - \frac{\left<e_i,a\right>}{\left<a,a\right>}a \right> = \left<y,e_i\right> = y_i \in \Z
\end{equation*}
for $i = 1, 2, \dots, n$.  This implies that $\Lambda ^* \subset \{x \in \Z^n: \left<x,a \right> = 0\} $, and thus 
\begin{equation*}
\Lambda^* = \{x \in \Z^n: \left<x,a \right> = 0\}
\end{equation*}

\end{proof}

\begin{proof}[Proof of Lemma \ref{GeneralDeterminantLemma}]
Suppose $a = (a_1, a_2, \dots, a_n)$ where $\{a_1, a_2, \dots, a_n\}$ is a set of relatively prime integers and let 
\begin{equation*}
\Lambda = P_a\left(\Z^n\right)
\end{equation*}
From Lemma \ref{ProjectionLatticeLemma} we know that $\Lambda^* = \{x \in \Z^n: \left<x,a\right> = 0\}$.  It is also well-known in the theory of lattices that for any lattice $\Lambda$, $\det(\Lambda)\det(\Lambda^*) = 1$.  (See, for example, Chapter VII section 7 of \cite{MR1940576}).  Thus, to prove our Lemma, we need only prove that 
\begin{equation*}
\det(\Lambda^*) = \sqrt{a_1^2+a_2^2+ \cdots +a_n^2}
\end{equation*}
We shall prove this, although we note that the proof is equivalent to solving problem number 4 in Chapter VII, section 2 of \cite{MR1940576}.  

First we let $N = a_1^2+a_2^2+ \cdots + a_n^2$ and define 
\begin{equation*}
\Lambda_1 = \{x \in \Z^d: \left<x, a\right> \equiv 0 \mod N\}
\end{equation*}
It is clear that $\Lambda_1$ is a sublattice of $\Z^n$, and since the numbers $a_1, a_2, \dots, a_n$ are relatively prime, the number of cosets of $\Lambda_1$ in $\Z^n$ is $n$.  From lattice theory (see Theorem 2.5 in Chapter VII, section 2 of \cite{MR1940576}), this implies that
\begin{equation*}
\det(\Lambda_1) = N \det(\Z^n) = N
\end{equation*}
Now suppose that $u_1, u_2, \dots, u_{n-1}$ is a basis for $\Lambda^*$.  We claim that $u_1, u_2, \dots, u_{n-1}, a$ is a basis for $\Lambda_1$. 

Indeed, if we write $x = \beta a +  \sum_{i=1}^{n-1}\beta_i u_i$ for integers $\beta, \beta_1, \beta_2, \dots, \beta_{n-1}$, then clearly $x \in \Z^n$ and 
\begin{equation*}
\left<x, a\right> = \beta\langle a,a\rangle + \sum_{i=1}^{n-1}\beta_i \left<u_i,a\right> = \beta N \equiv 0 \mod N
\end{equation*}
This implies that the integer span of $u_1, u_2, \dots, u_n, a$ is in $\Lambda_1$. 

Now pick $x \in \Lambda_1$.  Since $u_1, u_2, \dots, u_{n-1}$ is a basis for $\Lambda^*$ and $a$ is not in $\Lambda^*$, we know that the vectors $u_1, u_2, \dots, u_{n-1}, a$ span all of $\R^n$ so that we can write 
\begin{equation*}
x = \gamma a + \sum_{i=1}^{n-1} \gamma_i u_i
\end{equation*}
for some \emph{real numbers} $\gamma, \gamma_1, \gamma_2, \dots, \gamma_{n-1}$.  It remains to show that in fact $\gamma, \gamma_1, \dots, \gamma_n$ are in fact all integers.  Note that
\begin{equation*}
\left<x, a\right> = \beta\left<a,a\right> = \beta N \equiv 0 \mod N
\end{equation*}
This implies that $\beta \in \Z$.  Thus, we can see that
\begin{equation*}
\sum_{i=1}^{n-1} \beta_i u_i = x-\frac{\left<x,a\right>}{\left<a,a\right>}a = x-\beta a
\end{equation*}
where $\beta \in \Z$.  Thus, we can see that $\sum_{i=1}^{n-1} \beta_i u_i$ has integer coordinates and thus is in $\Lambda^*$.  By the definition of lattice basis, we must have $\beta_1, \beta_2, \dots, \beta_{n-1}$ all integers.  Thus we have proved that $u_1, u_2, \dots, u_{n-1}, a$ is a basis for $\Lambda_1$.

Finally, recall that the determinant of a lattice is the volume of the fundamental parallelepiped of that lattice.  Specifically, the volume of the $(n-1)$-dimensional parallelepiped defined by the vectors $u_1, u_2, \dots, u_{n-1}$ is the determinant of $\Lambda^*$ while the volume of the $n$-dimensional parallelepiped defined by the vectors $u_1, u_2, \dots, u_{n-1}, a$ is the determinant of the lattice $\Lambda_1$.  Let $P$ be the $(n-1)$-dimensional parallelepiped defined by vectors $u_1, u_2, \dots, u_{n-1}$.  Since $a$ is perpendicular to the span of $u_1, u_2, \dots, u_{n-1}$, we must have:
\begin{align*}
\text{Vol}(P)||a|| &= \text{Volume of parallelepiped defined by vectors}\{u_1, u_2, \dots, u_{n-1}, a\} \\
\det(\Lambda^*) ||a|| &= \det(\Lambda_1) \\
\det(\Lambda^*)\sqrt{N} &= N
\end{align*}
so that we have now shown that $\det(\Lambda^*) = \sqrt{N} = \sqrt{a_1^2+a_2^2+ \cdots + a_n^2}$.
\end{proof}

We can now see that in equation \eqref{EqnWithConstant}, we can take $c_i = ||v_i||$ for each $i=1, 2, \dots, k$.  This, along with Lemma \ref{LinearityLemma}, gives us the following more beautiful definition of the boundary function $b$:  For $A \subset \R^n$,
\begin{align*}
b(A)&  = \sum_{i=1}^k c_i(D_{u_i}(A) + D_{-u_i}(A)) \\
&= \sum_{i=1}^k ||v_i||\left(D_{u_i}(A) + D_{-u_i}(A) \right) \\
&= D_{Z}(A)
\end{align*}
where $Z$ is the zonotope from Theorem \ref{DiscToContTheorem}:
\begin{equation*}
Z = \sum_{i=1}^k \left(-\ell_i + \ell_i\right)
\end{equation*}

This completes the proof of Lemma 1.

\begin{Remark}\label{LimitingRemark}
We expect the solution $Z$ for the continuous isoperimetric problem on $\R^n$ using boundary function $b$ will give us the set $K$ for a PL graph whose edge isoperimetric inequality has a limiting solution.  Indeed, without loss of generality we can scale $K$ so that $\mu_n(K) = \mu_n(Z)$.  If $Z$ is not homothetic to $K$, we must have $b(Z) < b(K)$.  From equation \eqref{GeneralBoundaryEquations}, this implies that $\partial_e(\alpha Z) < \partial_e(\alpha K) $ for  $\alpha>0$ large enough.  This would be a contradiction if, for large enough $\alpha$, we ever had $|\Z^n \cap \alpha Z | = |\Z^n \cap \alpha K|$.  We note that, for all of our examples in Section \ref{KnownSection}, the PL graphs do have limiting solutions $K$ which are the same as the continuous solution $Z$.

\end{Remark}

\subsection{Proof of Continuous Problem}

Now, with the help of the Brunn-Minkowski inequality \cite{MR1898210}, we can prove Theorem \ref{DiscToContTheorem}.  The Brunn-Minkowski inequality can be stated as follows:
\begin{Theorem}[Brunn-Minkowski Inequality]
Suppose $A$ and $B$ are nonempty convex bodies in $\R^n$.  Then
\begin{equation*}
\mu_n(A+B)^{1/n} \geq \mu_n(A)^{1/n} + \mu_n(B)^{1/n}
\end{equation*}
with equality true if and only if $A$ and $B$ are nomothetic.
\end{Theorem}

This now allows us to now prove Theorem \ref{DiscToContTheorem}.

\begin{proof}[Proof of Theorem \ref{DiscToContTheorem}]
Suppose $G=(V,E)$ is a PL graph with edge segments $\ell_1, \ell_2, \dots, \ell_k$.  Let 
\begin{equation*}
Z = \sum_{i=1}^k \left(-\ell_i + \ell_i\right)
\end{equation*}
From the discussion above, the sets of optimum edge boundary $S_{k_n}$ have continuous counterparts $S_{k_n}'$ which are convex bodies.  Thus, to solve the continuous isoperimetric problem, we look for convex bodies $A \subset \R^n$ of fixed volume with minimum boundary $b(A)$.
  
Recall that we define 
\begin{equation*}
b(A) =D_Z(A) =  \lim_{\epsilon \to 0^+} \frac{\mu_n\left(A+ \epsilon Z \right) -\mu_n(A)}{\epsilon}
\end{equation*}
Using the Brunn-Minkowski inequality, we have
\begin{align}
b(A) &= \lim_{\epsilon \to 0} \frac{\mu_n\left(A+ \epsilon Z \right) - \mu_n(A)}{\epsilon}      \nonumber  \\
 & \geq \lim_{\epsilon \to 0} \frac{\left(\mu_n(A)^{1/n} + \epsilon \mu_n(Z)^{1/n}\right)^n - \mu_n(A)}{\epsilon}    \nonumber \\
&= n \mu_n(A)^{(n-1)/n} \mu_n(Z)^{1/n}  \label{BrunnMwithb}
\end{align}
with equality if and only if $A$ is a translate of a scalar multiple of $Z$.

We can also calculate 
\begin{align*}
b(Z) &= \lim_{\epsilon \to 0} \frac{\mu_n\left(Z+ \epsilon Z \right) - \mu_n(Z)}{\epsilon} \\
&= \lim_{\epsilon \to 0} \frac{\mu_n\left( (1+\epsilon) Z\right) - \mu_n(Z)}{\epsilon} \\
&= \lim_{\epsilon \to 0} \frac{(1+\epsilon)^n\mu_n(Z) - \mu_n(Z)}{\epsilon} \\
&= n \mu_n(Z)
\end{align*}

Substituting this into equation \eqref{BrunnMwithb}, we have
\begin{align*}
b(A) &\geq \frac{b(Z)}{\mu_n(Z)} \mu_n(A)^{(n-1)/n}\mu_n(Z)^{1/n} \\
\frac{b(A)}{b(Z)} & \geq \frac{ \mu_n(A)^{(n-1)/n}}{\mu_n(Z)^{(n-1)/n}} \\
\left(\frac{\mu_n(A)}{\mu_n(Z)}\right)^{1/n} &\leq \left(\frac{b(A)}{b(Z)} \right)^{1/(n-1)}
\end{align*}
with equality if and only if $A$ is a translate of a scalar multiple of $Z$.  

\end{proof}

We note that this proof is essentially the same as a proof for the Euclidean isoperimetric inequality, with $Z$ replaced by the Euclidean ball.

\section{Using the Continuous Technique for Known Isoperimetric Problems}\label{KnownSection}

\begin{Example}
Our technique for translating a discrete isoperimetric inequality into a continuous one can be applied to any PL graph.  One such graph that has been previously studied is the graph $(\Z^n, E_1)$ where $E_1$ denotes the set of edges which connects any pair of integer points whose $\ell_1$-distance is 1:
\begin{equation*}
E_1 = \{(x,y) \in \Z^n \times \Z^n: ||x-y||_1 = 1\}
\end{equation*}
where, as usual, for $x = (x_1, x_2, \dots, x_n), y = (y_1, y_2, \dots, y_n)$, we have
\begin{equation*}
||x-y ||_1 = \sum_{i=1}^n |x_i-y_i|
\end{equation*}
Bollob\'as and Leader studied the edge-isoperimetric inequality on this graph (and others) in \cite{MR1137765}.  While they used the corresponding continuous isoperimetric problem to solve their main discrete result in the paper, for the PL graph described above they used discrete methods.  In \cite{MR1137765}, Bollob\'as and Leader show that sets of minimal boundary in $(\Z^n, E_1)$  of size $s^n$ for $s \in \Z$ are boxes:
\begin{equation*}
[s]^n = \{(s_1, s_2, \dots, s_n): s_i \in \Z, 0 \leq s_i \leq s \text{ for }i = 1, 2, \dots, n\}
\end{equation*}
(Note that this easily shows that this graph satisfies the assumption listed in Subsection \ref{AssumptionsSubSection}.)

Using our continuous method, we see that the edges in this graph correspond to vectors $v_i= e_i$, $i=1, 2, \dots, n$, where $e_i$ is the $i$th standard basis vector whose entries are all 0 except the $i$th, which is 1.  Thus, letting $[-e_i,e_i]$ denote the line segment from $-e_i$ to $e_i$, our method would predict the sets of minimum boundary to have the shape of the zonotope
\begin{equation*}
\sum_{i=1}^n [-e_i,e_i]
\end{equation*}
which is, as expected, a box.
\end{Example}

\begin{Example}
Here, we consider the triangular lattice  in $\R^2$.  That is, we tile the plane with equilateral triangles, and from this we get a graph $T_2$ whose vertices are the vertices of each triangle, and edges are the edges of each triangle.

\begin{figure}[h]
\centering
\includegraphics[width=2in]{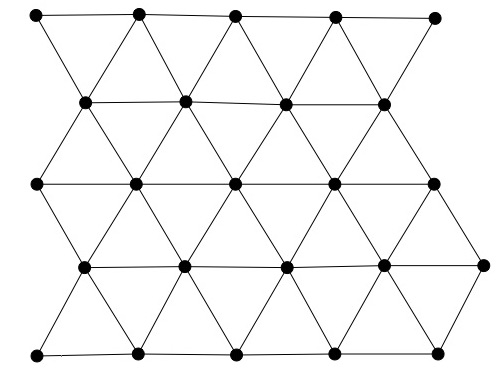}
\caption{\label{hex}}{Subgraph of $T_2$}
\end{figure}
We note that the edge-isoperimetric problem of this graph is of interest in the study of the emergence of the Wulff shape in the crystallization problem \cite{DavPioSte15}.

Using a linear transformation, this graph is isomorphic to a PL graph, so our technique will work for this graph.

 According to \cite{MR2035509}, the solutions are nested and for $|S|=1+3r(r+1)$ they are of the form
\[
B_r=\{v \in V \, : \, d(v_0,v) \leq r\}.
\]
(Again showing that this graph satisfies the assumption listed in Subsection \ref{AssumptionsSubSection}.)

From our technique, the limit zonotope is given by the sum of the edges:
\[
Z=\sum_{j=0}^5 [\vec{0},e^{ij \frac{\pi}{3}}]
\]
which is the regular hexagon, consistent with $B_r$.
\end{Example}

\section{Using the Continuous Technique on some New Graphs}\label{UnknownSection}

We can apply our continuous technique to any PL graph in order to give an idea of what the shapes of the sets of minimum edge boundary should look like for sets of large volume.  Here we apply this technique to two graphs whose edge-isoperimetric inequalities are not yet known.  For both of these examples, the solutions for the continuous case are fairly ``complicated,'' suggesting that finding these sets using discrete methods only would be quite difficult.

\begin{Example}

There exists a tessellation of $\R^4$ using 4-dimensional crosspolytopes; see \cite{MR0027148} for details.  One can translate this tessellation into a graph living in $\R^4$ as follows:  the 0-dimensional faces of these crosspolytopes become vertices and the 1-dimensional faces become edges.  The vertices of this graph are the points 
\begin{equation*}
(x_1, x_2, x_3, x_4) \in \Z^4  \quad \text{ such that } \quad x_1+x_2+x_3+x_4 \equiv 0 \mod 2
\end{equation*}
and the edges involving any vertex $v$ are of the form
\begin{align*}
&(v,v+e_1) \quad \text{ where } e_1 \text{ is any permutation of } (1,-1,0,0) \\
\text{ and }  &(v, v\pm e_2) \quad \text{ where } e_2 \text{ is any permutation of } (1,1,0,0) 
\end{align*}

In \cite{MR2035509}, Harper mentions that the isoperimetric inequality for this graph is unknown.  Note that this graph is isomorphic to a PL graph, so that we can use our technique to see that the limiting optimal shape should be the following zonotope:
\begin{equation*}
Z_C = \sum_{i=1}^{12} \left[\vec{0}, \pi_i(1,-1,0,0) \right] + \sum_{j=1}^6 \left[ \vec{0}, \phi_j(1,1,0,0)\right] + \sum_{k=1}^6 \left[\vec{0}, \psi_k(-1,-1,0,0)\right]
\end{equation*}
where the $\pi_i$s are the distinguishable permutations of $(1,-1,0,0)$, the $\phi_j$s are the distinguishable permutations of $(1,1,0,0)$, the $\psi_k$s are the distinguishable permutations of $(-1,-1,0,0)$, and $[\vec{0}, v]$ indicates the line segment from the origin to $v$.

Using the online version of polymake \cite{polymake} which can be found here: \newline
https://polymake.org/doku.php/boxdoc \newline
 we were able to find that the vertices of this zonotope are all coordinate permutations and sign combinations of (0,2,4,6) and this polytope has $f$-vector (192, 384, 240, 48).  (Thus, this limiting shape is apparently a truncated 24-cell \cite{Wiki24Cell}).  
 
 One might have guessed, given the definition of $Z_C$ as a zonotope, that it has facets corresponding to 
 \begin{equation*}
 \{x \in \R^4:\left<x,v\right> \leq c\} \quad \text{ where } v \text{ is a permutation of } (1,-1,0,0), (1,1,0,0), \text{ or } (-1,-1,0,0)
 \end{equation*}
 And it \emph{does} have all of those facets (according to polymake) with $c = 20$.  But it also has 24 other facets, corresponding to 
 \begin{align*}
 \{x \in \R^4: \left<x,u\right> \leq 12\}&  \quad \text{ where } u \text{ is a permutation of } (1,0,0,0) \text{ or } (-1,0,0,0) \\
 \text{and } \{x \in \R^4: \left<x,w\right> \leq 24\} & \quad \text{ where } w \text{ is one of } (\pm 1, \pm 1, \pm 1, \pm 1)
 \end{align*}
This shape is complicated enough that it likely would be quite difficult to find using only discrete methods.

\end{Example}

\begin{Example}
We also apply our technique to the edge-isoperimetric problem for a graph whose vertex-isoperimetric problem was recently solved in \cite{MR2946103}.  Here, we denote this graph by $(\Z^n, E_\infty)$.  Its vertices are $\Z^n$ and any two vertices whose $\ell_\infty$ distance is 1 have an edge between them:
\begin{equation*}
E_\infty = \{(x,y)\in \Z^n \times \Z^n: ||x-y||_\infty = 1\}
\end{equation*}
where, as usual, for $x = (x_1, x_2, \dots, x_n), y = (y_1, y_2, \dots, y_n)$, we have
\begin{equation*}
||x-y||_\infty = \max_{i=1, 2, \dots, n} |x_i-y_i|
\end{equation*}

This graph is clearly a PL graph such that the edges involving $v \in \Z^n$ are:
\begin{equation*}
(v, v+\epsilon) \quad \text{ where } \epsilon \in \{-1,0,1\}^n, \epsilon \not= \vec{0}
\end{equation*}

Thus, from our technique, the sets of minimum edge boundary should have the shape of
\begin{equation*}
Z_n =  \sum_{\substack{\epsilon \in \{-1,0,1\}^n \\ \epsilon \not= \vec{0}}}  \left[\vec{0},\epsilon\right]
\end{equation*}
where $\left[\vec{0},\epsilon\right]$ is, as usual, the line segment from the origin to $\epsilon$.

\end{Example}

These shapes can be seen in Figure \ref{zonotopes} (courtesy of Polymake \cite{polymake} and Sage \cite{sage}) for $n=2$ and $n=3$.

\begin{figure}[h]
\begin{center}
\subfigure[$n = 2$]{\label{zonotope2}\includegraphics[width=2 in]{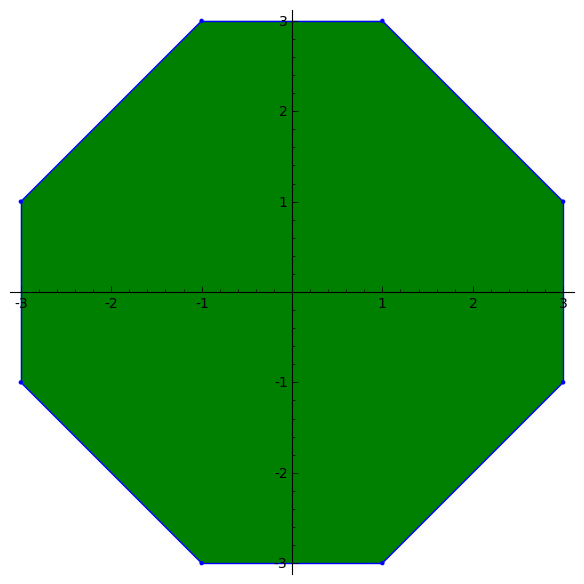}}
\hspace{1 in} \subfigure[$n = 3$]{\label{zonotope3}\includegraphics[width=2.7 in]{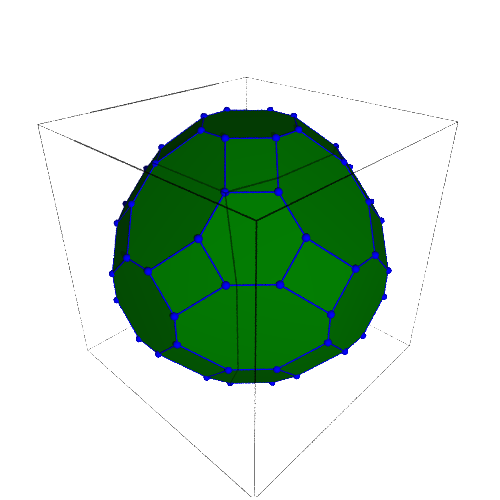}}
\end{center}
\caption{Optimal zonotopes for $(\Z^n, E_\infty)$ in dimensions 2 and 3.}
\label{zonotopes}
\end{figure}

From Polymake \cite{polymake}, we found that the f-vectors of $Z_2, Z_3$, and $Z_4$ were (8,8), (96, 144, 50), and (5376, 11328, 7312, 1360) respectively.

In the case of $n = 2$, it is not hard to argue that for a fixed boundary, the shape of largest volume lies inside a polygon defined by 8 facets:
\begin{align*}
\left<x,(1,0)\right> & \leq c_1 & \left<x,(0,1)\right> & \leq c_2  & \left<x,(1,1)\right> & \leq c_3 & \left<x,(1,-1)\right> & \leq c_4 \\
\left<x,(-1,0)\right> & \leq c_5 & \left<x,(0,-1)\right> & \leq c_6  & \left<x,(-1,-1)\right> & \leq c_7 & \left<x,(-1,-)\right> & \leq c_8
\end{align*}

Using precise boundary calculations and discrete volume calculations (from Pick's Theorem, which can be found in Section 2, Chapter VII of \cite{MR1940576}) one can show that the optimal sets in the discrete 2-dimensional case are indeed growing octagons, and thus are nested.  Moreover, as expected, the shape from picture \ref{zonotope2} appears periodically in the optimal discrete sets of growing volume (whenever it can be achieved with a particular discrete volume).  Thus, one might predict that the sets in higher dimensions are also nested. 

Frequently, when a graph is defined for any dimension $n$ such as this one is, and the sets of optimum boundary are nested, one can use the technique of compression to prove the discrete isoperimetric inequality \cite{MR2035509}.  This technique requires that sections of the optimal set in dimension $n$ which are perpendicular to the coordinate axes are optimal sets in dimension $n-1$.  Interestingly, this is \emph{not} the case here.  In the Proposition below, we show that sections of $Z_n$ which are perpendicular to the coordinate axes, and either through the origin or on the boundary \emph{are} in the shape of  $Z_{n-1}$.  However, Sage \cite{sage} can show us that already the section of $Z_3$ through the hyperplane $x = 3$ is \emph{not} $Z_2$ (since it is not an octagon), see Figure \ref{SectionFigure}.  
\begin{figure}[h]
\begin{center}
\includegraphics[width=2 in]{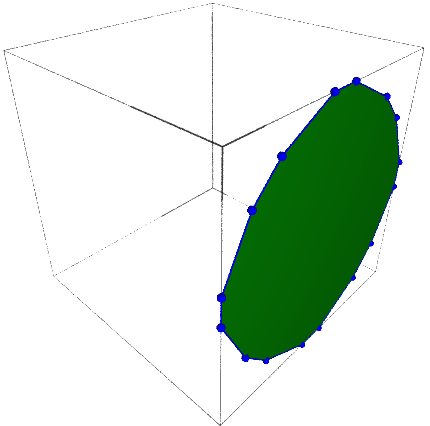}
\end{center}
\caption{Section of $Z_2$ through the plane $x = 3$}
\label{SectionFigure}
\end{figure}

\begin{Proposition}
\begin{enumerate}
\item Let $e_i$ be a standard basis vector and let $F_i$ be the face of $Z_n$ defined as follows:
\begin{equation*}
F_i = \{z \in Z_n: \left<e_i,z\right> = \text{max}\{\left<e_i,x\right>: x \in Z_n\}\}
\end{equation*}
Then $F_i$ is a translation of the zonotope $Z_{n-1}$.

\item Additionally, define $X$ to be the intersection of $Z_{n-1}$ with the hyperplane consisting of all points whose $i$th coordinate is 0:
\begin{equation*}
X = \{z \in Z_n: \left<e_i,x\right> = 0\}
\end{equation*}
Then $X$ is the set $3Z_{n-1}$ embedded into that hyperplane.  
\end{enumerate}
\end{Proposition}

\begin{proof}
Let $x \in F_i$ for the set $F_i$  as defined above.  For ease of notation, say $i=1$.   Define
\begin{align*}
E_{-1} &=\{ \epsilon = (\epsilon_1, \epsilon_2, \dots, \epsilon_n) \in \{-1,0,1\}^n : \epsilon_1 = -1\} \\
E_{0} &= \{\epsilon = (\epsilon_1, \epsilon_2, \dots, \epsilon_n) \in \{-1,0,1\}^n : \epsilon_1 = 0\}\\
E_{1} &= \{\epsilon = (\epsilon_1, \epsilon_2, \dots, \epsilon_n) \in \{-1,0,1\}^n : \epsilon_1 = 1 \}\\
\end{align*}

 Since $x \in Z_n$, we know that we can write
\begin{equation}\label{FacetEquation}
x= \sum_{\substack{\epsilon \in \{-1,0,1\}^n  \\ \epsilon \not= \vec{0}}} \lambda_\epsilon \epsilon
\end{equation}
where  $0 \leq \lambda_\epsilon \leq 1$.  Since $x \in F_1$, we know that the first coordinate of $x$ must be as large as possible; that is, it must be $3^{n-1}$.  This implies that
\begin{equation}\label{FacetEquation2}
x = \sum_{\epsilon \in E_1} \epsilon + \sum_{\epsilon \in E_0} \lambda_\epsilon \epsilon
\end{equation}
where $0 \leq \lambda_\epsilon \leq 1$.  That is, in the expression \eqref{FacetEquation}, if $\epsilon \in E_1$ then $\lambda_\epsilon = 1$; if $\epsilon \in E_{-1}$  then $\lambda_\epsilon = 0$, and if $\epsilon \in E_0$ then $\lambda_\epsilon$ may be anything between 0 and 1.  Thus, if we define $Z$ as $Z_{n-1}$ embedded into the hyperplane of $\R^n$ consisting of all vectors with first coordinate 0, then this shows 
\begin{equation*}
F_1 \subset Z + v
\end{equation*}
where $v =  \sum_{\epsilon \in E_1} \epsilon $.  

It is also clear that any $x$ that can be expressed as in equation \eqref{FacetEquation2} must also be in $F_1$.  This implies that $Z+v \subset F_1$, and we have proved our first claim.  

Define $X$ as above,  let $x \in X$, and again for ease assume $i = 1$.  Then we can write 
\begin{align*}
x &= \sum_{\substack{\epsilon \in \{-1,0,1\}^n  \\ \epsilon \not= \vec{0}}} \lambda_\epsilon \epsilon \\
&= \sum_{\epsilon \in E_0} \lambda_\epsilon \epsilon +  \sum_{\epsilon \in E_{-1}} \lambda_\epsilon \epsilon +  \sum_{\epsilon \in E_1} \lambda_\epsilon \epsilon \\
&=  \sum_{\epsilon \in E_0} \lambda_\epsilon \epsilon +  \sum_{\epsilon \in E_{-1}} \lambda_\epsilon (-e_1 + (\epsilon+e_1)) +  \sum_{\epsilon \in E_1} \lambda_\epsilon (e_1+(\epsilon-e_1)) \\
\end{align*}
Since we know that $\left<x,e_1\right> = 0$, we must have 
\begin{equation}\label{SectionEquation}
\sum_{\epsilon \in E_0} \lambda_\epsilon \epsilon +  \sum_{\epsilon \in E_{-1}} \lambda_\epsilon (-e_1 + (\epsilon+e_1)) +  \sum_{\epsilon \in E_1} \lambda_\epsilon (e_1+(\epsilon-e_1))= 
\sum_{\epsilon \in E_0} \lambda_\epsilon \epsilon +  \sum_{\epsilon \in E_{-1}} \lambda_\epsilon (\epsilon+e_1) +  \sum_{\epsilon \in E_1} \lambda_\epsilon (\epsilon-e_1)
\end{equation}
Note that each $\epsilon \in E_0$ also appears in the other two sums on the right side of equation \eqref{SectionEquation}.  Thus, grouping the three coefficients for the same $\epsilon \in E_0$, we see that
\begin{equation*}
x= \sum_{\epsilon \in E_0} c_\epsilon \epsilon
\end{equation*}
and we must have $0 \leq c_\epsilon \leq 3$.   Thus, if we define $Z$ as $Z_{n-1}$ embedded into the hyperplane of $\R^n$ consisting of all vectors with first coordinate 0, then this shows 
\begin{equation*}
X \subset 3Z
\end{equation*}
Now suppose that $x \in 3Z$.  That is, we can write
\begin{equation*}
x= \sum_{\epsilon \in E_0} c_\epsilon \epsilon
\end{equation*}
where $0 \leq c_\epsilon \leq 3$.  Define:
\begin{align*}
\lambda_{\epsilon,0} & = 
\begin{cases} c_\epsilon & \text{ if } c_\epsilon \leq 1 \\
1 & \text{ if } c_\epsilon > 1
\end{cases} \\
\lambda_{\epsilon,1} &= 
\begin{cases} 0 & \text{ if } c_\epsilon \leq 1 \\
\frac{c_\epsilon -1}{2} & \text{ if } c_\epsilon > 1
\end{cases} \\
\lambda_{\epsilon,-1} &= 
\begin{cases} 0 & \text{ if } c_\epsilon \leq 1 \\
\frac{c_\epsilon -1}{2} & \text{ if } c_\epsilon > 1
\end{cases}
\end{align*}

Then we can see that clearly $\left<e_1,x \right> = 0$ and 
\begin{equation*}
x= \sum_{\epsilon \in E_0} c_\epsilon \epsilon = \sum_{\epsilon \in E_0} \left(\lambda_{\epsilon,0} \epsilon + \lambda_{\epsilon,1} \left(e_1+\epsilon\right) + \lambda_{\epsilon,-1} \left(-e_1+\epsilon \right)\right)
\end{equation*}
where $0 \leq \lambda_{\epsilon,i} \leq 1$ for $i = -1, 0, 1$.  This shows that $x \in X$.  Thus, we have shown that $3Z \subset X$ and we have proved our second claim.
\end{proof}

The complicated structure of the sets $Z_n$ again suggests that it would be difficult to find these optimal sets using discrete methods alone.

\section*{Acknowledgments}  
The authors would like to thank Alexander Barvinok for his helpful comments.

\bibliographystyle{plain}
\bibliography{EdgeIsoperimetry}

\end{document}